\documentclass[11pt,english,oneside]{amsart}

\usepackage[usenames,dvipsnames]{xcolor}  

\usepackage[T1]{fontenc}
\usepackage{amsmath,amsthm,indentfirst}
\usepackage{amssymb}
\usepackage{amsfonts,ifpdf}

\ifpdf
\usepackage[pdftex,pdfstartview=FitH,pdfborderstyle={/S/B/W 1}]{hyperref}
\else
\usepackage[hypertex]{hyperref}
 \fi
\usepackage{amscd}
\usepackage[latin1]{inputenc}
\DeclareMathAlphabet{\mathpzc}{OT1}{pzc}{m}{it}
\DeclareFontFamily{OT1}{rsfs}{}
\DeclareFontShape{OT1}{rsfs}{n}{it}{<-> rsfs10}{}
\DeclareMathAlphabet{\mathscr}{T1}{rsfs}{n}{it}

\usepackage{babel}
\usepackage{verbatim}
\usepackage{amsthm}
\usepackage{amstext}
\usepackage{amssymb}
\usepackage{amsbsy}

\numberwithin{equation}{section}

\usepackage{ifpdf}
\usepackage{amsrefs}

\usepackage{enumerate}
\usepackage{amsmath}
\usepackage{amscd}
\usepackage{amsfonts}
\usepackage{graphicx}
\usepackage[all]{xy}
\usepackage{verbatim}
\usepackage{gensymb}
\usepackage{mathrsfs}

\numberwithin{equation}{section}
\numberwithin{figure}{section}

\theoremstyle{plain}
\newtheorem{thm}{Theorem}[section]

\newtheorem{prop}[thm]{Proposition}
\newtheorem{cor}[thm]{Corollary}
\newtheorem{lem}[thm]{Lemma}
\theoremstyle{remark}
\newtheorem*{claim*}{Claim}

\newcommand{\cO}{\mathcal{O}}

\newcommand{\cS}{\mathcal{S}}

\newcommand{\bG}{\mathbb{G}}
\newcommand{\bL}{\mathbb{L}}

\newcommand{\bR}{\mathbb{R}}
\newcommand{\bZ}{\mathbb{Z}}
\newcommand{\bQ}{\mathbb{Q}}

\newcommand{\bN}{\mathbb{N}}
\newcommand{\bH}{\mathbb{H}}

\newcommand{\gog}{\mathfrak{g}}
\newcommand{\goh}{\mathfrak{h}}

\newcommand{\gor}{\mathfrak{r}}
\newcommand{\goS}{\mathfrak{S}}

\newcommand{\SL}{\operatorname{SL}}

\newcommand{\Lie}{\operatorname{Lie}}
\newcommand{\vol}{\operatorname{vol}}

\newcommand\norm[1]{\left\|#1\right\|}

\newcommand\set[1]{\left\{#1\right\}}
\newcommand\pa[1]{\left(#1\right)}

\newcommand\av[1]{\left|#1\right|}

\newcommand{\height}{{\rm ht}}

\newcommand{\ra}{\rightarrow}

\newcommand{\onto}{\xymatrix{\ar@{>>}[r]&}}
\newcommand{\da}[4]{\xymatrix{#1 \ar@<.5ex>[r]^{#2} \ar@<-.5ex>[r]_{#3} & #4}}

\newif\ifdraft\drafttrue

\makeatother

\newcommand{\manote}[1]{\marginpar{\color{magenta}\tiny [MHAG] #1}}

\newcommand\xappa\kappa
\newcommand\yota\iota
\newcounter{consta}
\renewcommand{\theconsta}{{\xappa_{\arabic{consta}}}}
\newcounter{constb}

\newcounter{constc}[section]
\renewcommand{\theconstc}{{c_{\arabic{constc}}}}
\newcommand{\consta}{\refstepcounter{consta}\theconsta}

\newcommand{\constc}{\refstepcounter{constc}\theconstc}

\newcommand{\Ad}{\operatorname{Ad}}

\newcommand{\beq}{\begin{equation}}
\newcommand{\eeq}{\end{equation}}
\newcommand{\gfrak}{\mathfrak{g}}
\newcommand{\Sfrak}{\mathfrak{S}}

\newcommand{\disc}{{\rm disc}}
\newcommand{\Rmetric}{{\rm dist}}
\newcommand{\temp}{{M}_0}
\newcommand{\mht}{\operatorname{mht}}

\begin{document}

\title{On Effective Equidistribution for Quotients of $\SL(d, \mathbb R)$}

\author[M. Aka]{Menny Aka}
\author[M. Einsiedler]{Manfred Einsiedler}
\author[H. Li]{Han Li}
\author[A. Mohammadi]{Amir Mohammadi}

\begin{abstract}
We prove the first case of polynomially effective equidistribution of closed orbits of 
semisimple groups with nontrivial centralizer. The proof relies on uniform spectral gap, 
builds on, and extends work of Einsiedler, Margulis, and Venkatesh. 
\end{abstract}

\address{M.A. Departement Mathematik\\
ETH Z\"urich\\
R\"amistrasse 101\\
8092 Zurich\\
Switzerland}

\email{mennyaka@math.ethz.ch}
\thanks{M.A.\ acknowledges the support of ISEF, and SNF Grant 200021-152819. }

\address{M.E. Departement Mathematik\\
ETH Z\"urich\\
R\"amistrasse 101\\
8092 Zurich\\
Switzerland}

\email{manfred.einsiedler@math.ethz.ch}
\thanks{M.E.\ acknowledges the support of the SNF Grant 200021-152819 and 200020-178958. }

\address{H.L. Department of Mathematics and Computer Science, Wesleyan University, 265 Church Street, Middletown, CT 06457}
\email{hli03@wesleyan.edu}
\thanks{H.L.\ acknowledges support by Simons Foundation (426090) and NSF (DMS 1700109).}

\address{A.M.: Department of Mathematics, University of California, San Diego, CA 92093}
\email{ammohammadi@ucsd.edu}
\thanks{A.M.\ acknowledges support by the NSF (DMS 1724316, 1764246, 1128155) and Alfred P.~Sloan Research Fellowship.}
 
\maketitle

\section{Introduction}
Rigidity results for dynamics of group actions on homogeneous spaces, i.e., quotients of Lie groups by discrete subgroups of finite covolume, have been a subject of great interest with several striking results and applications.
Indeed the solution of Oppenheim's conjecture by Margulis in \cite{Margulis-Oppenheim-proof} opened up a new chapter 
in the dialogue between homogeneous dynamics and number theory. The landmark results in \cite{Ratner-proc-measure,Ratner-acta-measure,Ratner-invent-solvable,Ratner-measure-rigidity,Ratner-bull-distribution}
of Ratner on unipotent dynamics became quite quickly the engines to many interesting applications
of homogeneous dynamics. 
The proofs of these rigidity results use techniques from ergodic theory and are often not quantitative. It is a challenging problem, with several applications, to provide effective and quantitative accounts of these rigidity results.

In \cite{EMV}  a polynomially effective equidistribution theorem for closed orbits of semisimple
group $H$ is proven under the assumption that the Lie algebra $\mathfrak{h}$ of $H$ has trivial centralizer in the Lie algebra $\mathfrak{g}$ of the ambient group $G$. As explained in \cite{EMV} this centralizer assumption does not seem to be truly essential to the method. We consider a first case where similar results can be obtained in the presence of a one-dimensional centralizer.

Let $k,l\in \bN$ and assume $\bG$ is a $\bQ$-form of ${\rm SL}_{k+l}$ that splits over $\bR.$ 
Consider a $\bQ$-embedding 
\begin{equation}\label{eq:Q-embedding}
\rho:\bG\ra \SL_N
\end{equation} 
for some $N\in \bN$. Set $G=\bG(\bR)\cong\SL_{k+l}(\bR).$ 
By a theorem of Borel and Harish-Chandra $\Gamma:=\rho^{-1}\pa{\SL_N(\bZ)}\cap G$ 
is a lattice in $G$. We define $X=G/\Gamma.$ 
Throughout the paper we divide matrices in ${\rm Mat}_{k+l}$ into blocks consisting of the first $k$ or last $l$ 
rows and columns. Moreover, we consider 
the algebraic group
$$
\bH=\SL_k\times \SL_l=\left[
\begin{array}{c|c}
\SL_k &0 \\ \hline
0 &\SL_l
\end{array}\right]
 $$
over $\bR;$ let $H=\bH(\bR)$.

As the centralizer of $H$ is not trivial, $H$-orbits may lie far from any given compact set, e.g.\ this 
happens for the $\bQ$-split group $\bG=\SL_{k+l}$ and orbits $Ha\SL_{k+l}(\bZ)\subset\SL_{k+l}(\bR)/\SL_{k+l}(\bZ)$ 
for a large $a\in C_G(H)$. 
This is obviously an obstruction to equidistribution, and we take this possibility into account
via a height function $\height(\cdot)$ on $X$ whose definition is given in \eqref{eq:height} and the function
$\mht(\mathcal Q)=\inf\{\height(x): x\in \mathcal Q\}$ for subsets $\mathcal Q\subset X.$

A closed subgroup $S\subset G$ containing $H$ we will call an \emph{intermediate} subgroup; we will show in~\S\ref{sec:subgroups} that there are only finitely many such subgroups.
Similar to~\cite{EMV}, we define the notion of volume of a closed $S$-orbit using a Haar measure on $S$. More precisely 
we fix a Haar measure $m_S$
on $S$ and define the \emph{volume} $\vol(Sx)$ of a closed $S$-orbit $Sx$ to be $m_S(F)$ where $F\subset S$ is a Borel fundamental domain for the quotient map $S\to Sx$. In contrast, $\mu_{Sx}$  denotes the normalized Haar probability measure on the orbit $Sx$. 

\begin{thm}\label{thm:main}
Assume that $(k,l)\neq (2,2)$ and fix $G,H,\Gamma$ as above. There exist $d\in \bN$ and  $\consta\label{vol exponent main},\consta\label{height exponent} >0$ depending only on $G$ and $H$ and a constant $V_0=V_0(G,H,\Gamma)>0$ such that for all $V>V_0$ there exists an intermediate subgroup $H\subset S\subset G$ such that for any $f\in C_c^{\infty}(X)$ we have  
$$\av{\mu_{Hx_0}(f)-\mu_{Sx_0}(f)}\ll \cS_d(f)\mht(Hx_0)^{\ref{height exponent}} V^{-\ref{vol exponent main}}\text{ and }{\rm vol}(Sx_0)\leq V,$$
for any closed orbit $Hx_0$,
where the implicit constant depends only on $G,H,\Gamma$ and $\cS_d$ is a Sobolev norm (defined in \S \ref{sec:sob}).
\end{thm}

The general strategy of the proof is similar to that of~\cite{EMV}. We use spectral
gap to show in an effective way that most points are generic 
in an effective manner for the dynamics of a unipotent subgroup. 
In fact, we are relying on uniform spectral as provided by the so-called property ($\tau$),
which in general is the combination of results of Selberg \cite{selberg1965estimation} for congruence quotients of 
the split form of $\operatorname{SL}_2$, Jaquet-Langlands \cite{jacquet2006automorphic} 
for other forms of $\operatorname{SL}_2$,
extensions of this by Burger and Sarnak \cite{burger1991ramanujan}, Kazhdan's property (T) 
from \cite{kazhdan1967connection}, and Clozel's work \cite{clozel2003demonstration}. We refer also to \cite[\S4]{EMMV}
for more details and a more dynamical proof of that fact. In the context of this paper
Kazhdan's property (T) is sufficient once $\min(k,\ell)\geq 3$. 

Using the effectively generic points and an effective version of the polynomial divergence property,
that also played a big role in the work of Margulis and Ratner mentioned before, 
we effectively produce almost invariance under new elements transversal to $H.$
This is then upgraded to establish
almost invariance under a subgroup $S\supsetneq H.$
The special choice of $H,$ in particular, the fact that the centralizer of $H$ is one dimensional simplifies the proof
in several places. This makes the possibilities of $S$ quite restricted, see~\S\ref{sec:subgroups}, which allows us
to use well-known facts regarding effective equidistribution of horospherical orbits in our proof.
The special case at hand also allows us to utilize known results from nondivergence of unipotent flows and get
a much simplified version of a closing lemma --  an effective closing lemma is one of the main technical ingredients in~\cite{EMV}, 
which is used to handle the intermediate orbits.
Indeed, we show that  $Sx_0$ is closed
unless the orbit $Hx_0$ is far in the cusp, see~\S\ref{sec:DM}.

The case of $(k,l)=(2,2)$ is slightly more complicated due to the presence of further intermediate subgroups $S$ (see~\S\ref{sec:subgroups}) and we avoid it in the current paper. 

We thank the anonymous referee for her or his remarks which helped to improve the presentation of the paper.

\section{Notation and Preliminaries}
As this work builds heavily on \cite{EMV} we borrow much notation and conventions from loc.\ cit.
\subsection{Constants and their dependency and Landau's notation}
The notation $A \ll B$, meaning ``there exists a constant $\constc\label{sample-c}>0$ so that $A \leq \ref{sample-c} B$'',
will be used; the implicit constant $\ref{sample-c}$ is permitted to depend on
$\bG$ and $\rho$, but (unless otherwise noted) not on anything else. We write~$A\asymp B$ if~$A\ll B\ll A$.
We will use~$\ref{sample-c},\constc,\ldots$ to denote  positive constants 
depending on~$\bG$ and~$\rho$ (and their numbering is reset at 
the end of a section).
If a constant (implicit or explicit) depends on another parameter or only on a certain
part of~$(\bG,\rho)$, we will make this clear by writing e.g.~$\ll_\epsilon$,~$\constc(N)$, etc.

We will use~$\ref{vol exponent main}, \ref{height exponent},\consta,\ldots$ to denote positive constants depending only on $\dim\bG.$
We also adopt the $\star$-notation from~\cite{EMV}:
we write $B=A^{\pm\star}$ if $B=\constc A^{\pm\consta\label{sample-kappa}}.$ 
Similarly one defines
$B\ll A^\star,$ $B\gg A^\star$.
Finally we also write~$A\asymp B^\star$ if~$A^\star\ll B\ll A^\star$ (possibly with different exponents).

\subsection{Setup}\label{sec:Lie-notation}\label{sec:setup}
 Much of the notation below will depend on the choice of $k,$ $l$ and $N$ in~\eqref{eq:Q-embedding}
which are fixed throughout the paper. 

We recall the definition of congruence lattices in our setting. A congruence subgroup of $\SL_N(\bR)$ is a subgroup commensurable to $\SL_N(\bZ)$ containing a principal congruence subgroup, i.e.\ a kernel of the reduction map  $\SL_N(\bZ)\ra\SL_N(\bZ/D\bZ)$ for some $D\in \bN$. 
We assume that $\Gamma\supset\rho^{-1}(\Gamma')$ where $\Gamma'$ is a congruence subgroup of $\SL_N$ and $\rho$ was defined in~\eqref{eq:Q-embedding}. 
By the same argument as in~\cite[\S 1.6.1]{EMV} Theorem \ref{thm:main} also holds for arithmetic subgroup at the cost of allowing the exponent to depend on $\Gamma$. 

Given an element $g\in G$ we let $|g|=\max\{\|g\|_\infty,\|g^{-1}\|_\infty\}.$  
We fix a Euclidean norm $\norm{\cdot}$ on $\gog:={\rm Lie}(G)$ such that $\norm{[v_1,v_2]}\leq \norm{v_1}\norm{v_2}$.  The embedding $\rho:\bG\ra \SL_N$ induces a $\bQ$-structure on $\gog$ and we may choose a $\Gamma$-stable lattice $\gog_\bZ$ such that $[\gog_\bZ,\gog_\bZ]\subset \gog_\bZ$.
We also let $\Rmetric(\cdot,\cdot)$ denote a right invariant Riemannian metric on $G.$ 

The choice of the inner product on $\gog$ induces a normalization of the Haar measure on any closed subgroup of $G$ and therefore a notion of \emph{volume} for orbits  of these subgroups in $X$. Given a subgroup $P$ and a point $x\in X$ we denote this volume measure as $\operatorname{d}\!{\rm vol}$ and the volume of the orbit $Px$ as ${\rm vol}(Px)$. In contrast, $\mu_{Px}$  denotes the normalized Haar probability measure on the orbit $Px$. 
Let us write $g.f(x):=f(g^{-1}x)$ for $g\in G$ and $f\in C(X)$ and $x\in X$. Similarly, for any measure $\nu$ on $X$  we let $g_*\nu$ be the measure defined by $g_*\nu(A)=\nu(g^{-1}A)$ for any Borel set $A\subset X$.

We choose $\mathfrak r$, a particular ${\Ad}(H)$-invariant complement of $\goh={\rm Lie}(H)\subset \gog$, by letting
$\mathfrak r=\mathfrak r_0\oplus \mathfrak r_1$ where, using block notation, $\mathfrak r_1=\mathfrak r_1^+\oplus \mathfrak r_1^-$ and 

$$
\mathfrak r_0={\rm span}\left[
\begin{array}{c|c}
l\cdot I_k &0 \\ \hline
0 &-k\cdot I_l
\end{array}\right],
\quad 
\mathfrak r_1^+=\left[
\begin{array}{c|c}
0 &* \\ \hline
0 &0
\end{array}\right],
\quad 
\mathfrak r_1^-=\left[
\begin{array}{c|c}
0 &0 \\ \hline
* &0
\end{array}\right].
$$
Let $u_k(t)\in \SL_k(\bR)$ denote the unipotent element

$$
\left[\begin{array}{cccc}
1&\frac{t^1}{1!}&\cdots &\frac{t^{k-1}}{(k-1)!}\\
0&1&\cdots &\frac{t^{k-2}}{\pa{k-2}!}\\
\vdots & &\ddots &\vdots \\
0&0&0 &1\\
\end{array}\right]=\exp 
\left[\begin{array}{cccc}

0&t&\cdots &0\\
\vdots & \ddots&\ddots &\vdots  \\
0&\cdots&0 &t\\
0&ֿ\cdots&0 &0\\
\end{array}\right]
$$
and let $$U=\set{u(t):t\in \bR},\qquad u(t)=\left[
\begin{array}{c|c}
u_k(t) &0 \\ \hline
0 &u_l(t)
\end{array}\right].$$ 
For $\mathfrak s\in\set{\mathfrak r_1,\mathfrak r_1^+,\mathfrak r_1^-}$ we put ${\rm Fix}_U(\mathfrak s):=\{w\in\mathfrak s: \Ad(u)w=w\text{ for all } u\in U\}.$

We also define the one-parameter subgroup 
$$
a_t=
\left[
\begin{array}{c|c}
e^{lt}I_k &0 \\ \hline
0 & e^{-kt}I_l
\end{array}\right],\qquad A=\{a_t:t\in \bR\}
$$
whose Lie algebra is $\mathfrak r_0$.

For a diagonalizable element $a$ we define the {\em expanding horospherical} subgroup 
\[
W^+_G(a)=\set{g\in G: a^nga^{-n}\ra e,\text{ as } n\ra-\infty},
\] 
and the {\em contracting horospherical} subgroup $W^-_G(a)=W^+_G(a^{-1})$.

Put $W:=W^+_G(a_1)$ for $a_1\in A$ as above, and note that 
\begin{equation}\label{eq:def-W}
W=
\left[
\begin{array}{c|c}
I_k & * \\ \hline
0 & I_l
\end{array}\right].
\end{equation} 
Finally we let $P^\pm$ denote the connected subgroups of $G$ with 
$$
\Lie(P^{+})=\Lie(H)\oplus \mathfrak r_1^{+},\quad\Lie(P^{-})=\Lie(H)\oplus \mathfrak r_1^{-}.
$$
\subsection{Height, discriminant and volume}\label{sec:ht-disc}
Given a lattice $\mathfrak s_{\bZ}$ in a vector space $\mathfrak s$ 
and a subspace $\mathfrak l$ that intersects $\mathfrak s_{\bZ}$ in a lattice
we define the covolume (or discriminant) of $\mathfrak l$ by setting
\begin{equation}\label{eq:discriminant}
\mathrm{covol}(\mathfrak l):= \disc(\mathfrak l):=\|p_{\mathfrak l}\| 
\end{equation}
where $p_{\mathfrak l}$ is a primitive vector in 
\begin{equation}\label{eq:primitive-vector}
\textstyle{\bigwedge^{\dim\mathfrak l}}\mathfrak l\cap\textstyle{\bigwedge^{\dim\mathfrak l}}\mathfrak s_{\bZ}.
\end{equation}

For an element $g\in G$ we say that a subspace $\mathfrak l\subset\gog$ is {\em $g$-rational} if $\mathfrak l\cap {\rm Ad}(g)(\gog_\bZ)$ is a lattice in $\mathfrak l$. 
It will be called simply \emph{rational} if it is $g$-rational for $g=e$.
Given a $g$-rational subspace $\mathfrak l$ we define the covolume of $\mathfrak l$ using~\eqref{eq:discriminant}
with $\mathfrak s=\gog$ and $\mathfrak s_{\bZ}=\Ad(g)\gog_{\bZ}$.

For a $\bQ$-subgroup $\mathbb L$ of $\mathbb G$
we put $\disc(\mathbb L):=\disc(\Lie(\mathbb L)).$ 

Recall that the lattice $\gog_\bZ$ is $\Gamma$-stable.
Hence we may define the \emph{height} of a point $x\in X$ by
\beq\label{eq:height}
\height (x)=\sup\set{\norm{\Ad(g).v}^{-1}:x=g\Gamma,\ v\in\gog_\bZ\setminus\{0\}}.
\eeq
The height of $x$ in~$\SL_N(\bR)/\SL_N(\bZ)$ is defined similarly.
This defines the term $\mht(Hx_0)$
appearing in the statement of Theorem~\ref{thm:main}.
 
Let $\mathfrak S(R)=\set{x\in X: \height (x)\leq R}$. 
By Mahler's compactness criterion, the sets $\set{\mathfrak S(R):R>0}$ are all compact and 
$\cup_{R>0}\mathfrak S(R)=X$.

\subsection{Spectral input}\label{ss:spectral-input}
Let us denote by $L$ the group generated by $U$ and its transpose; it is isomorphic to $\SL_2(\mathbb R)$ and we call it the \emph{principal} $\SL_2(\mathbb R)$.
Also let $P^\pm$ be defined as in~\S\ref{sec:setup}. 

We will use the following as a blackbox: The representations of $L$, the principal $\SL_2(\mathbb R)$, on 
\[
L_0^2(\nu)=\bigl\{f\in L^2(X,\nu):\textstyle\int f\operatorname{d}\!\nu=0\bigr\},
\]
are $1/\temp$-tempered (i.e.\ the matrix coefficients of the $\temp$-fold tensor product
are in~$L^{2+\epsilon}(\SL_2(\mathbb R))$ for all~$\epsilon>0$), where
$\nu$ is the $S$-invariant probability measure on a closed $S$-orbit with $S=H,P^\pm$ or  $S=G$. 

Since $H\subset S$ for all choices of $S$ above, $1/\temp$-temperedness follows directly in the case when $H$ has property (T), see \cite[Thm.~1.1--1.2]{OhDuke},
and in the general case we may apply property $(\tau)$ in the strong form, see~\cite{ClozelTau}, \cite{GMO} and \cite[\S6]{EMV}.

\subsection{Sobolev norms}\label{sec:sob}
We now recall the definition of a certain family of Sobolev norms and their main properties (see ~\cite[\S 3.7]{EMV}). For any integer $d\geq0$ we let $\cS_d$ be the Sobolev norm on $X$ defined by
\begin{equation}\label{sobolev}
\cS_d(f)^2=\sum_{\mathcal{D}}\|\height(\cdot)^d\mathcal{D}f\|_2^2,
\end{equation}
where $f\in C_c^\infty (X)$ and the sum is taken over all $\mathcal D\in U(\gog)$, the universal enveloping algebra of $\gog$, which are monomials in a chosen basis of $\gog$ of degree at most $d.$ 
We will need the following properties of $\cS_d$. There exists a constant $\consta\label{sob-prod}$ such that for any $d\geq \ref{sob-prod}$ and any $g\in G$ and $f\in C_c^\infty(X)$, we have
\begin{enumerate}
\item[($\cS$-1)] For any $g\in G$ and $f\in C_c^\infty(X)$ we have 
\[
\cS_d(g.f)\ll_d |g|^{3d}\cS_d(f),
\]

\item[($\cS$-2)] For any $f\in C_c^\infty(X)$ we have 
\[
\|f\|_\infty\ll_d \cS_d(f),
\]

\item[($\cS$-3)] For any $g\in G$ and $f\in C_c^\infty(X)$ we have 
\[
\|g.f-f\|_\infty\ll_d {\Rmetric}(e,g)\cS_d(f),
\]

\item[($\cS$-4)] For any $f_1,f_2\in C_c^\infty(X)$ we have 
\[
\cS_d(f_1f_2)\ll_d\cS_{d+\ref{sob-prod}}(f_1)\cS_{d+\ref{sob-prod}}(f_2).
\]

\item[($\cS$-5)] Let $\nu$ and $\temp$ be as in \S\ref{ss:spectral-input}. We have 
\beq
\label{sobnormmc}\Big| \langle u(t).f_1, f_2 \rangle_{L^2(\nu)} - 
	\nu( f_1)\nu (\bar f_2) \Big|
 \ll_d (1+|t|)^{-\frac{1}{2\temp}} \cS_{d}(f_1) \cS_d(f_2).
\eeq
\end{enumerate}

For a discussion of the Sobolev norm, the reason for introducing the factor $\height(\cdot)^d,$ and the proofs of the above properties we refer to~\cite[\S5]{EMV}.

Let $\cS_d$ be as above and let $\epsilon>0$. We say that a measure $\sigma$ is $\epsilon$-almost invariant under $g\in G$ (w.r.t.\ $\cS_d$) if 
\[
|\sigma(g.f)-\sigma(f)|\leq\epsilon\;\cS_d(f)\quad\text{for all $f\in C_c^\infty(X).$}
\] 
We say that $\sigma$ is $\epsilon$-almost invariant under a subgroup 
$L\subset G$ if it is $\epsilon$-almost invariant under all $g\in L$ with $|g|\leq 2.$
Similarly, given $w\in\gog$ we say $\sigma$ is $\epsilon$-almost invariant under $w$ if
$\sigma$ is $\epsilon$-almost invariant under $\exp(tw)$ for all $|t|\leq 2.$

\section{Structure of intermediate subgroups}\label{subsec:Inter} 
\label{sec:subgroups}
 Any finite-dimensional representation of $H$ decomposes into irreducible sub-representations as $H$ is semisimple.
In order to study the connected intermediate subgroups we may work with the Lie algebra of $G$, see \cite[\S7]{Borel-AlgGrBook}. Consider the adjoint representation of $H$ on $\Lie(G)$. It decomposes as 
 $$
 \Lie(H)\oplus \mathfrak r_1^+\oplus \mathfrak r_1^-\oplus \mathfrak r_0.
 $$ 
Indeed, it is easily verified that each of these factors are sub-representations and a dimension count shows that it is a complete decomposition.  The analysis of the possible intermediate subgroups follows simply from noting  that for any intermediate closed subgroup with $H\subset S\subset G$,  $\Lie S$ will be a subrepresentation, which is also a Lie subalgebra, of $\Lie(G)$. 

\begin{prop}\label{Prop:subgroups not 2,2}
Fix $(k,l)$ with $\max\set{k,l}\geq 3$ and let $S$ be a closed connected 
subgroup with $H\subset S\subset G$. Then $$S\in\set{ H,P^+,P^-,AH,AP^+,AP^-,G}.$$ 
$$AP^{+}=\exp(\Lie(H)\oplus \mathfrak r_1^{+}\oplus \mathfrak r_0),\, AP^{-}=\exp(\Lie(H)\oplus \mathfrak r_1^{-}\oplus \mathfrak r_0)$$ $$AH=\exp(\Lie(H)\oplus \mathfrak r_0).$$ 
\end{prop}

\begin{proof}
Note that $\mathfrak r_1^+$ and $\mathfrak r_1^-$ are both irreducible and are dual to each other. If $\max\set{k,l}\geq 3$,  then we claim that the representations  $\mathfrak r_1^+$ and $\mathfrak r_1^-$ are non-isomorphic.
Indeed, if they were isomorphic, then they will be isomorphic also as a representation of the larger block, say of  $\SL_k<H$. Note that as an $\SL_k$ representation $\mathfrak r_1^+$ is a direct sum of the standard representation of $\SL_k$ on $\bR^k$ and $\mathfrak r_1^-$ is a direct sum of its dual. If they where isomorphic then the standard representation on $\SL_k$ on $\bR^k$ is isomorphic to its dual, which is a contradiction when $k\geq 3$ (e.g.\ because  
${\rm diag}(t,\cdots,t,t^{-(k-1)})\in \SL_k$ and its inverse cannot be conjugated to each other when $k\geq 3$). 

The proposition now follows as the possible  subrepresentations of $\Lie(G)$ which contain $\Lie(H)$ correspond exactly to the Lie algebras of the groups listed above. 
\end{proof}

For the cases $(k,l)\in \set{\pa{2,2},\pa{2,1},\pa{1,2}}$ we have that $\mathfrak r_1^+$ and $\mathfrak r_1^-$ are isomorphic as representations of $H$. 
When $k=l=2$ this isomorphism gives rise to a family of subgroups, which are isomorphic to ${\rm Sp}(4)$. This case will probably also yield to the methods of this paper, but requires a special treatment in each step, and therefore we avoid it in the current paper. In contrast, we have:

\begin{prop}\label{Prop:subgroups 2,1}
Proposition \ref{Prop:subgroups not 2,2} holds also when $k=2,l=1$ (or $k=1,l=2)$.
\end{prop}

\begin{proof} Fix an  isomorphism $\phi:\mathfrak r_1^+\ra \mathfrak r_1^-$  and let   
\[
\mathfrak s_p:=\set{\pa{v,p \phi(v)}:v\in \mathfrak r_1^+}.
\]
for $p\in \bR$. The proposition will follow once we will show 
that the $H$-subrepresantation $\Lie(H)\oplus \mathfrak s_p$ or 
$\Lie(H)\oplus \mathfrak s_p\oplus \mathfrak r_0$ for $p\in \bR\setminus\{0\}$ 
are never Lie subalgebras. This follows just by calculations of Lie brackets. Indeed, for concreteness, let $e_1,e_2$  (resp.\ $f_1,f_2$) denote the standard basis of $\mathfrak r_1^+$ (resp.~$\mathfrak r_1^-$) and fix $\phi$ to be the isomorphism sending $\alpha e_1+\beta e_2\mapsto -\beta f_1+\alpha f_2$. Now, a direct calculation shows that the commutator of $(e_1,p\phi(e_1))\in \mathfrak s_p$ and $(e_2,p\phi(e_2))\in \mathfrak s_p$ is a non-trivial element of $\mathfrak r_0$ (whenever $p\neq 0$).
 Another calculation shows that the Lie algebra generated by $\mathfrak r_0$ and  $\mathfrak s_p$ contains $\mathfrak r_1^+\oplus\mathfrak r_1^-\oplus \mathfrak r_0$, so the Lie subalgebra generated by $\Lie(H)\oplus \mathfrak s_p$ is always $\Lie(G)$.
\end{proof}

\section{Applying nondivergence of unipotent flows}\label{sec:DM} 
When $\text{rank}_\bQ \bG=0$ the quotient $X$ is compact. In this case the addition of the height function 
in the definition of the Sobolev norm in not needed and several other analytic arguments become simpler. 
In particular, this section is only important in the case that $\text{rank}_\bQ \bG>0$.   

Let us recall the definition of certain functions $d_\alpha:G\rightarrow\mathbb R.$ 
These functions were considered by Dani and Margulis in~\cite{DM} in order to 
study the recurrence properties of unipotent flows on homogeneous spaces.  
Let $\mathbb S$ be a maximal $\bQ$-split $\bQ$-torus of $\bG$. Let $\mathbb P\supset \mathbb S$
be a minimal $\mathbb Q$-parabolic subgroup and let $\Delta$ be the associated simple roots relative 
to $\mathbb S$, see \cite[Sec.\ 12]{BT}. 
For $\alpha\in\Delta$, let $\mathbb P_\alpha$ be 
the corresponding maximal $\mathbb Q$-parabolic subgroup. 
Let $\mathbb U_\alpha=R_u(\mathbb P_\alpha)$ be the unipotent radical and let $\mathfrak{u}_\alpha$ denote the Lie algebra of 
$\mathbb U_\alpha.$ Put $\ell_\alpha:=\dim\mathfrak{u}_\alpha$ and let 
$\vartheta_\alpha=\wedge^{\ell_\alpha}{\rm Ad}$ denote the $\ell_\alpha$-th exterior power of the 
adjoint representation. Note that $\wedge^{\ell_\alpha}\mathfrak u_\alpha$ defines a 
$\mathbb Q$-rational one dimensional subspace of $\wedge^{\ell_\alpha}\mathfrak g.$ 
Fix a unit vector $v_\alpha\in\wedge^{\ell_\alpha}\mathfrak u_\alpha.$ 
Note that if $g\in \mathbb P_\alpha(\bR),$ then 
\[
\vartheta_\alpha(g)v_\alpha=\det({\rm Ad}(g)|_{\mathfrak u_\alpha})v_\alpha.
\] 
Define $d_\alpha:G\rightarrow\mathbb R$ by 
\[
d_\alpha(g)=\|\vartheta_\alpha(g) v_\alpha\|\quad\text{for all $g\in G$}.
\]  
For each $\alpha\in\Delta$ put $\mathbb P^{(1)}_\alpha=\{g\in \mathbb P_\alpha: \vartheta_\alpha(g)v_\alpha=v_\alpha\}.$
Put $P_\alpha=\mathbb P_\alpha(\mathbb R)$ and $P^{(1)}_\alpha=\mathbb P^{(1)}_\alpha(\mathbb R).$
Since $P^{(1)}_\alpha$ is a $\bQ$-group without any $\bQ$-characters, it follows from a theorem of Borel and Harish-Chandra, that $P^{(1)}_\alpha\Gamma/\Gamma$ is a closed orbit with a finite $P^{(1)}_\alpha$-invariant measure.

\begin{thm}[Cf.~\cite{DM}, Theorem 2]\label{thm:DM-parabolic}
There exist a finite subset $\Xi\subset\mathbb G(\mathbb Q),$
and some $R_0>0$ with the following property. 
For every $x=g\Gamma\in X$ there exists $T_x$ so that one of the following holds:
\begin{enumerate}
\item $|\{|t|\leq T: u(t)g\Gamma\in \mathfrak S(R_0)\}|\geq (1-2^{-20})T$ for all $T>T_x$.
\item There exist $\lambda\in \Gamma\Xi$ and $\alpha\in\Delta$ such that 
$g^{-1}Ug\subset \lambda P^{(1)}_\alpha\lambda^{-1}$ and moreover
$
 d_\alpha(g\lambda)<1.
$
\end{enumerate} 
\end{thm}  

We note that if $\mathbb G$ has $\bQ$-rank zero, then $X$ is compact and the height function is bounded. In particular,
the first case in the theorem and its corollary below hold trivially.  

\begin{cor}\label{cor:closing-lem}
Let $y=g\Gamma$ be so that $Hy$ is a closed orbit.
Then one of the following holds:
\begin{enumerate} 
\item  $\mu_{Hy}\pa{Hy\setminus \mathfrak S(R_0)}\leq 2^{-10}.$
\item There exist $\lambda\in \Gamma\Xi$, $\alpha\in\Delta$ and $h\in H$ with $\av{h}\leq 2$ such that 
$g^{-1}Hg\subset \lambda P^{(1)}_\alpha\lambda^{-1},$ and
$
 d_\alpha(hg\lambda)<1.
$ 
\end{enumerate}
\end{cor}

\begin{proof}
Let $h\in H$ with $|h|\leq 2$ be so that $hg\Gamma\in Hy$ is a generic point 
for the action of $U=\{u(t):t\in\mathbb R\}$ in the sense of Birkhoff ergodic theorem; that is
\beq\label{eq:Birk-ghGamma}
\lim_T\frac{1}{T}\int_0^T f(u(t)hg\Gamma)\operatorname{d}\!t=\int f\operatorname{d}\!\mu_{Hy}.
\eeq 
for all $f\in C_c(X).$

We consider two possibilities. First let us assume that  
Theorem~\ref{thm:DM-parabolic}(1) holds true for $x=hg\Gamma.$
Then the conclusion (1) of Corollary~\ref{cor:closing-lem} holds by~\eqref{eq:Birk-ghGamma}.

Therefore, let us assume that Theorem~\ref{thm:DM-parabolic}(2) holds true for $x=hg\Gamma$ and we show that the  conclusion (2) of Corollary~\ref{cor:closing-lem} must hold true.
Then there exist some $\lambda\in\Gamma\Xi$ and some $\alpha\in\Delta$ so that $d_\alpha(hg\lambda)<1$
and  
\[
g^{-1}h^{-1}Uhg\subset \lambda P^{(1)}_\alpha\lambda^{-1}.
\]

We claim that $g^{-1}Hg\subset \lambda P^{(1)}_\alpha\lambda^{-1}.$ 
Since $\lambda\in\Gamma\Xi$ and $\Xi\subset \bG(\bQ)$ we have that the orbit $\lambda P^{(1)}_\alpha\lambda^{-1}\Gamma/\Gamma$
is a closed orbit. Hence, 
\beq\label{eq:U-orbit-Q}
\overline{g^{-1}h^{-1}Uhg\Gamma/\Gamma}\subset\lambda P^{(1)}_\alpha\lambda^{-1}\Gamma/\Gamma
\eeq
However, by~\eqref{eq:Birk-ghGamma} we have 
\[
\overline{g^{-1}h^{-1}Uhg\Gamma/\Gamma}=g^{-1}Hg\Gamma/\Gamma.
\]
This together with~\eqref{eq:U-orbit-Q} implies that $g^{-1}Hg\subset \lambda P^{(1)}_\alpha\lambda^{-1}$
as we claimed.
\end{proof}

\begin{lem}[Mass in the cusp]\label{lem:non-div}
There exist constants $\consta\label{exp:escape},\consta\label{escape:height}$ depending only on $G$ and  $\constc\label{escape}$ 
such that for any periodic $H$-orbit $Hx,$ we have 
\[
\mu_{Hx}\pa{Hx\setminus \mathfrak S(R)}\leq \ref{escape}{\mht}(Hx)^{\ref{escape:height}}R^{-\ref{exp:escape}}.
\]
\end{lem}

\begin{proof}

 Let $z\in Hx$ be the point of the smallest height, namely ${\rm ht}(z)={\mht}(Hx)$. 
By the Birkhoff ergodic theorem, there exists some $h\in H$ with $|h|\leq 1$ 
such that  
\begin{equation}
\label{eq:birk-nondiv}
\lim_{T\to\infty}\frac{1}{T}\Big|\left\{0<t<T: {\rm ht}(u(t)hz)> L \right\}\Big|=\mu_{Hx}\pa{Hx\setminus \mathfrak S(L)}.
\end{equation}
for all $L\in\mathbb N.$

Let $hz=g\Gamma$ for some element $g\in G$. 
Let $\mathfrak s$ be a $g$-rational subspace of $\gog$.  As was done in~\cite[App.~B]{EMV}, define the function 
\[
\psi_{\mathfrak s}(t)={\rm covol}({\rm Ad}(u_t)\mathfrak s)^2.
\]
Note that $\psi_{\mathfrak s}$ is a polynomial whose degree is bounded in terms of $\dim G$ only. 
On the other hand, since $|h|\leq 1$ we have $\psi_{\mathfrak s}(0)={\rm covol}(\mathfrak s)^2\gg {\rm ht}(z)^{-m}$ for some 
absolute constant $m$ depending on $N.$ 
We have, by~\cite[Thm.~5.2]{KM}, for any $T>0$ 
\[
\frac{1}{T}\Big|\left\{0<t<T: {\rm ht}(u_thz)> R  \right\}\Big|\ll \Bigl(\frac{{\rm ht}(z)^{m}}{R}\Bigr)^{\ref{exp:escape}} =\frac{{\mht}(Hx)^{\ref{escape:height}}}{R^{\ref{exp:escape}}}.
\]
where $\ref{exp:escape}$ depends on degree of $\psi_{\mathfrak s}$
and $\ref{escape:height}:=m\ref{exp:escape}.$
This together with~\eqref{eq:birk-nondiv} implies the claim.
\end{proof}

Given a closed orbit $Hx,$ define
\beq\label{eq:def-R0}
R_{Hx}:=\tfrac{2^{{20}/{\ref{exp:escape}}}{\mht}(Hx)^{{\ref{escape:height}}/{\ref{exp:escape}}}}{\ref{escape}^{1/\ref{exp:escape}}}.
\eeq
This choice in view of Lemma~\ref{lem:non-div} implies the following.
\beq\label{eq:cpct-R0}
\mu_{Hx}\pa{Hx\setminus \mathfrak S(R_{Hx})}\leq2^{-20}.
\eeq

\section{From generic points to new almost invariants}
Let $\nu$ and $\temp$ be as in \S\ref{ss:spectral-input};
we continue to denote by $\mu$ the $H$-invariant  probability measure on $Hx_0.$
Let $M=20\temp$ and let $T\geq 1$ be a parameter.  Following \cite{EMV} we define for a function $f$ 
\beq\label{eq:erg-aver}
D_{T,\nu}(f)(x)=\frac{1}{\pa{T+1}^M-T^M}\int_{T^M}^{\pa{T+1}^M}f(u(t)x)\operatorname{d}\!t-\int_{X}f\operatorname{d}\!\nu.
\eeq
We write $D_T$ for $D_{T,\mu}.$
A point $x\in X$ is called $T_0$-\emph{generic} for the measure $\nu$ w.r.t.\ a Sobolev norm $\cS$  
if for all integers $n>T_0$ and all $f\in C_{c}^{\infty}(X)$ we have 
\begin{equation}\label{eq:generic}
\av{D_{n,\nu}(f)(x)}\leq n^{-1}\cS(f).
\end{equation}
Furthermore, a point $x\in X$ is called  $[T_0,T_1]$-\emph{generic} if (\ref{eq:generic}) is satisfied for all $n\in [T_0,T_1]$.

Let $\operatorname{d}\!w$ be the Lebesgue measure on $W\cong\bR^{kl}$, see~\eqref{eq:def-W}, and for any $\tau>0$
put 
\[
W[\tau]:=\{w\in W: \|w\|_{\infty}\leq\tau \}.
\]

\begin{lem} [Cf.~\cite{EMV}, \S 9.1] \label{lem:generic-almost-inv}
\begin{enumerate}
\item For $d\gg1$, depending only on $G$, the $\nu$-measure of the points that are not $T_0$-generic  for $\nu$ with respect to $\cS_{d}$ is $\ll T_0^{-1}$.
\item There exists $d'>d$ depending on $d$ and $G$ with the following property. Suppose 
\[
|w_*\mu(f)-\mu(f)|\ll \epsilon \cS_{d}(f)\quad\text{for all $w\in W[\tau].$}
\]
Then there exists some $\consta\label{k:gen-ni-tau}$ 
so that the proportion of points 
 $(w,x)\in W[\tau]\times X$ such that $wx$ is {\em not} $[T_0,\epsilon^{-\ref{k:gen-ni-tau}}]$-generic w.r.t.\ $\cS_{d'}$ 
is $\ll T_0^{-1}.$ 
\end{enumerate}
\end{lem}
\begin{proof}
First note that using~\eqref{sobnormmc} we have the following estimate on the $L^2$-norm of $D_{T,\nu}(f)$. 
\begin{equation}\label{l2}
\int_X |D_{T,\nu}(f)|^2 \operatorname{d}\!\nu \ll T^{-4} \cS_{d}(f)^2.
\end{equation}
The deduction of part (1) from~\eqref{l2} is identical to that of \cite[Prop.\ 9.1]{EMV}.

For (2), consider the integral
\begin{eqnarray}\label{dtf}
\frac{1}{|W[\tau]|}\int_{W[\tau]}\int_X |D_T(f)(wx)|^2 \operatorname{d}\!\mu(x) \operatorname{d}\!w.\end{eqnarray}
The inner  integral of \eqref{dtf} satisfies
\[
\int_X |D_Tf(wx)|^2 \operatorname{d}\!\mu(x)=w_*\mu(|D_T(f)|^2) \ll \epsilon \cS_d(|D_T(f)|^2) +\mu(|D_T(f)|^2).
\]
By properties of Sobolev norm in \S\ref{sec:sob}
\[
\cS_d(|D_T(f)|^2)\ll T^{\star d}\cS_{d+\ref{sob-prod}}(f)^2.
\]
Combining this with \eqref{l2} we have
\[
\int_X |D_Tf(wx)|^2 \operatorname{d}\!\mu(x)\ll \epsilon T^{\star d}\cS_{d+\ref{sob-prod}}(f)^2+T^{-4} \cS_{d}(f)^2.
\]
Therefore, if we choose $T$ so that $\epsilon T^{\star d}=T^{-4},$ then
\begin{equation}\label{part2}
\frac{1}{|W[\tau]|}\int_{W[\tau]}\int_X |D_T(f)(wx)|^2 \operatorname{d}\!\mu(x) \operatorname{d}\!w\ll T^{-4}\cS_{d+\ref{sob-prod}}(f)^2.
\end{equation}
The deduction of part (2) from \eqref{part2} is identical to that of part (1), see also \cite[Prop.\ 9.1--9.2]{EMV}.
\end{proof}

The following lemma provides us with generic points which differ in 
directions transversal to $H.$ The proof is based on a pigeonhole principle argument.  
We note that in this lemma the existence of the centralizer~$\mathfrak r_0$
starts to play a more significant role.

\begin{lem}[Cf.~\cite{EMV}, Prop.~14.1]\label{lem:separation}
There exist $\consta\label{k:ht-vol}$ and $\consta\label{leaf:separation}$ with the following property. 
Let $Hx_0$ be a closed orbit so that
\beq\label{eq:low}
\vol(Hx_0)\geq{\mht}(Hx_0)^{\ref{k:ht-vol}}.
\eeq
Then there exist $w\in \mathfrak{r}\setminus\{0\}$ and $x,y\in \goS (R_{H{x_0}})\cap Hx_0$ so that the following hold. 
\begin{enumerate}
\item $\norm{w}\leq ({\vol}(Hx_0))^{-\ref{leaf:separation}}.$ Moreover, if we decompose $w=w_0+w_1$ into
a sum of $w_0\in\mathfrak r_0$ and $w_1\in\mathfrak r_1,$ then $\|w_1'\|\gg \|w_1\|$ where
\[
w_1=w_1'+w_1''\in{\rm Fix}_U(\mathfrak r_1)^\perp\oplus{\rm Fix}_U(\mathfrak r_1)
\] 
and the decomposition is with respect to the Euclidean structure on $\mathfrak r_1$ 
which is induced by $\norm{\cdot}$. 
\item $\exp(w)x=y$ 
\end{enumerate}
Further, given $T_0$ large enough, $x$ and $y$ can be chosen to be $T_0$-generic. 
\end{lem}

\begin{proof}
Let $T_0>1$, and let $E'$ be the set of $T_0$-generic points. 
Put $E:=E'\cap\goS(R_{Hx_0}).$ In view of Lemma~\ref{lem:generic-almost-inv}(1) and 
the choice of $R_0,$ see~\eqref{eq:def-R0}, we have $\mu(X\setminus E)\leq 2^{-10}$ assuming $T_0$
is sufficiently large depending on $G, H$ and $\Gamma.$ 

For any $\delta>0$ let $\mathfrak r_\delta$ (resp.\ $\mathfrak h_\delta$) denote the ball of radius $\delta$ in $\mathfrak r$ (resp.~$\mathfrak h$) around the origin with respect to the norm $\norm{\cdot}$ on $\mathfrak g$. 
 Throughout the proof we will put more and more restrictions on $\delta$. To begin with, let $\delta>0$ be smaller than $1/20$ of the minimum of the injectivity radii at all $z\in\mathfrak S(R_{Hx_0})$.
This constraint amounts to an inequality of the form
\begin{equation}\label{Rstarinlemma}
\delta\ll R_{Hx_0}^{-\star}.
\end{equation}
Moreover, assume $\delta$ is small enough so that the map
$(r,v)\in\mathfrak r_\delta\times\mathfrak h_\delta\mapsto\exp(r)\exp(v)$ is a diffeomorphism onto its image in $G$. Therefore, for any $z\in \goS (R_{Hx_0})$ 
the natural map from $\pi_z:\mathfrak r_\delta\times\mathfrak h_\delta\to X$ defined by $\pi_z(r,v):=\exp(r)\exp(v)z$ is a diffeomorphism. 
Let $\Omega=\exp(\mathfrak h_{2\delta}).$ Define the following function on $X$
\begin{equation}\label{eq:desityfunc}
\phi(z)=\frac{1}{\vol(\Omega)}\int_\Omega\chi_E(hz)\operatorname{d}\!\vol(h).
\end{equation}
We have $\int_X\phi(z)d\mu=\mu(E)\geq1-2^{-10}.$ 
Put 
\beq\label{eq:def-F}
F:=\{z\in E: \phi(z)>0.99\}.
\eeq 
Then $\mu(F)\geq0.9.$

For $\delta$ chosen as above define $\mathsf B(z,\delta):=\pi_z(\mathfrak r_\delta\times\mathfrak h_\delta)$. 
We may cover ${F}$ by $\ll\delta^{-\dim G}$-many sets of the form $\mathsf B(z,\delta)$ with $z\in F$  and with
finite multiplicity depending only on $G$.
Using the pigeonhole principle we have the following. 
So long as 
\begin{equation}\label{bluething} 
\vol(Hx_0)\gg \delta^{\dim H-\dim G},
\end{equation}
there exist $x',y'\in F\cap\mathsf B(z,\delta)$ for some $z\in F$
so that $x'\neq hy'$ for any $h\in\Omega.$ 

We now want to perturb $x', y'$ to obtain elements of $E$ that satisfy the above claimed properties. 
First, note that since $\mathfrak r_1^+$ and $\mathfrak r_1^-$ are irreducible representations of $H,$ 
there is a constant $\iota>0$ such that
\begin{equation}\label{eq:conjgeneric}
\vol\{h\in\Omega: {\|{\Ad}(h)(r)'\|}\leq \iota {\|r\|}\}<\vol(\Omega)/2
\end{equation}
for all $r\in\mathfrak r_1$ where $r'$ is the component of $r$ in ${\rm Fix}_U(\mathfrak r_1)^\perp.$
Now, if we apply the Implicit Function Theorem and use the fact that 
$\phi(x')>0.99$ and $\phi(y')>0.99$ we can find $h_1,h_2\in\Omega$ such that 
\begin{itemize}
\item $h_1x',h_2y'\in E$ and 
\item $h_2y'=\exp(w)h_1x'$ where $w\in\mathfrak r$ and $\|w\|\ll\delta,$ 
\item  $\|w_1'\|\gg\|w_1\|.$ 
\end{itemize}

Therefore,  $w$, $x=h_1x',$ and $y=h_2 y'$ satisfy the conclusion of the proposition  so long as we choose
$\delta=\vol(Hx_0)^{-\star}$ so that \eqref{bluething} holds. In light of our assumption \eqref{eq:low} 
it is possible to satisfy this and our earlier requirement in \eqref{Rstarinlemma}.
\end{proof}
 
\begin{lem}\label{lem:not-cent} Let $w\in\mathfrak r$ be the ``difference'' found in Lemma~\ref{lem:separation}. Then 
 $\exp(w)\notin C_G(H)$.
\end{lem}

\begin{proof}
Let us write $\exp(w)x=hx$ for some $h\in H$   and $x=g\Gamma$.
We also let $\bH$ be the connected, simply connected, algebraic group 
such that $\bH(\bR)=g^{-1}Hg$. Note that $\bH$ is defined over $\bQ$ as $Hx$ is a closed orbit. With this notation we have $\exp(-w)hg\Gamma=g\Gamma,$ 
and the claim is equivalent to showing $g^{-1}\exp(-w)g\notin C_\bG(\bH)(\bR)$.  

Assume this is not the case. Then 
\[
\gamma=g^{-1}\exp(-w)gg^{-1}hg\in \bL:=C_\bG(\bH)\cdot \bH.
\] 
Note that $C_\bG(\bH)$ is one-dimensional. We define a set of characters on $\bL$,  $\Delta=\set{\chi_1,\chi_2}$, as follows. First note that for $\ell\in \bL(\bR)$, $g\ell g^{-1}$ has a block structure, that is, it has the form
$\left[
\begin{array}{c|c}
A &0 \\ \hline
0 &B
\end{array}\right]
 $.
We define  $\set{\chi_1(\ell),\chi_2(\ell)}$ to be the determinants of the diagonal blocks of  $g\ell g^{-1}$. 
This determines $\Delta$. 
Further note, that since $\bH$ is semisimple, any character on  $\bH$ is trivial. Moreover, 
for $a\in \bL(\bR)$ we have that $a\in \bH(\bR)$ if and only if $a$ is in the kernels of $\chi_i$, $i=1,2$.

Furthermore, we have that  $\chi_2=\chi_1^{-1}$ and that $\Delta$ is stable under the Galois group 
${\rm Gal}(\overline \bQ/\bQ)$ so either $\chi_1$ and $\chi_2$ are both defined over $\bQ$ or over a real  quadratic extension (as the centralizer splits over $\bR$). 
In both cases, the integrality of $\gamma$ implies that $\chi_i(\gamma)$ is either $1$ or uniformly bounded away from $1$. 
Indeed in the second case $\chi_i(\gamma)\in\cO^\times\subset \bR$ where $\cO$ is an order in a real quadratic extension.
And if a unit $u$ in a quadratic field is close to $1$, then $u+u^{-1}$ is close to $2$ and an integer which implies that $u=1$.
In particular this argument is independent of the quadratic extension.

On the other hand, as $\chi_i$ are trivial on $\bH$,  we have for $i=1,2$ 
\begin{equation*}
\chi_i(\gamma)=\chi_i(g^{-1}\exp(-w)g).
\end{equation*}
By definition the characters $\chi_i$ are defined by conjugating elements of $ \bL=C_\bG(\bH)\cdot \bH$
to $C_G(H)H$ and taking the determinants of one of the blocks.
In other words, we are taking the determinants of the blocks of the matrix $\exp(-w)$, which is
at distance $\ll ({\vol}(Hx_0))^{-\ref{leaf:separation}}$ from the identity. However, with the above
it follows that 
\[
\chi_i(g^{-1}\exp(-w)g)=1\quad \text{for $i=1,2$}.
\] 
Since the kernel of these characters is $\bH(\bR)$, this contradict the fact that $g^{-1}\exp(-w)g\notin \bH(\bR)$.
\end{proof}

We now use the effective ergodic theorem, Lemma~\ref{lem:generic-almost-inv}(1), 
and the above results to prove the following.

\begin{lem}[Cf.~\cite{EMV}, Prop.\ 10.1]\label{lem:eff-extrs-inv} Assume that 
\eqref{eq:low} holds.
There exists some $v\in{\rm Fix}_U(\mathfrak r_1)$ with $\|v\|=1$ 
so that
\[
|{\exp(tv)}_{*}\mu(f) - \mu(f)|  \leq \vol(Hx_0)^{-\star} \cS(f), \quad \text{for all }\; |t|\ll 1.
\]
\end{lem}

\begin{proof}
We will show that there exists some $v\in{\rm Fix}_U(\mathfrak r_1)$ with $\|v\|=1$ 
so that 
\beq\label{eq:almost-inv-v}
|{\exp(v)}_{*}\mu(f) - \mu(f)|  \leq \vol(Hx_0)^{-\star} \cS(f).
\eeq 
The claim for all $|t|\ll1$
will follow from this using conjugation by elements in $H$ as in Lemma \ref{lem:iterate}, or alternatively, by an argument 
as in the proof of~\cite[Prop.\ 10.1]{EMV}.

Let $T_0$ and $x,y\in Hx_0$ be as in Lemma~\ref{lem:separation}. 
In particular, $x$ and $y$ are $T_0$-generic, 
$y=\exp(w)x$ with $w\in\mathfrak r,$ $\|w\|\ll\vol(Hx_0)^{-\star},$ and $\|w_1'\|\gg\|w_1\|$. 
Recall also that by Lemma~\ref{lem:not-cent} we have $w_1\neq 0.$

Let us write $w=w_0+w_1''+w_1'$ as in Lemma~\ref{lem:separation}. 
Then 
\[
\Ad(u(t))w=w_0+w''_1+\Ad(u(t))w'_1.
\]
Therefore, there exists $T_1\gg \|w'_1\|^{-\star}$ and a polynomial 
$p:\mathbb R\to{\rm Fix}_U(\mathfrak r_1)$ with degree $\ll N$  
and $\sup\{\|p(t)\|: 0\leq t\leq1\}=\|p(1)\|=1$ so that 
\beq\label{eq:poly-p}
\Ad(u(t))w=p(t/T_1)+O(\|w\|^\star)
\eeq
for all $t\in[0,2T_1]$.
This is the polynomial divergence property of unipotent flows relying on the fact that $\Ad(u(t))w$ is a $\mathfrak{g}$-valued
polynomial whose terms of highest degree belong to~${\rm Fix}_U(\mathfrak{r})$.

Let now $n> T_0$ be so that $T_1\in [n^M,(n+1)^M]$ where $M$ is as in~\eqref{eq:erg-aver}.
Then by~\eqref{eq:generic} we have 
\[
|D_n(f)(z)|\leq n^{-1}\cS(f)\quad \text{for $z=x,y.$}
\]

In view of~\eqref{eq:poly-p}, property ($\cS$-3), and the fact that $f\in C_c^\infty(X)$ we have
\[
f(u(t)y)=f(\exp(p(t/T_1))u(t)x)+O(\|w\|^{\star})\cS(f)
\]
for all $t\in[0,2T_1]$.
Moreover, for any $t\in [n^M,(n+1)^M]$ we have $|t-T_1|\ll T_1^{1-1/M}.$
Hence all together we get
\[
f(u(t)y)=f(\exp(p(1))u(t)x)+O(T_1^{-1/M}\cS(f))+O(\|w\|^{\star})\cS(f).
\]
Therefore, $\mu(f)=\mu(\exp(p(1)).f)+O(T_1^{-1/M}+\|w\|^{\star})\cS(f).$\\
Since $\|p(1)\|=1,$ $T_1\gg \|w'_1\|^{-\star},$ and $\|w\|\ll\vol(Hx_0)^{-\star}$ we get~\eqref{eq:almost-inv-v}. 
\end{proof}

\section{Effective generation of a bigger group}\label{sec:generation}
We continue to use the previous notation. 
Let us first recall the following.

\begin{prop}[Cf.~\cite{EMV}, Proposition 8.1]\label{prop:eff-gen}
Let $\cS_d$ be a fixed Sobolev norm. Suppose that $\mu$ is $\epsilon$-almost invariant 
w.r.t.\ $\cS_d$ under $w\in\exp(\gor_1)$ with $\|w\|=1$. Then there exists $\consta\label{k:alm-inv-grp}>0$ 
so that $\mu$ is $c(d)\epsilon^{\ref{k:alm-inv-grp}}$-almost invariant w.r.t.\ $\cS_d$ 
under at least one of the groups $P^+$, $P^-$, or $G$. 
\end{prop}

\begin{proof}
We note that~\cite[Prop.\ 8.1]{EMV} is proved in the general 
setting that applies to our situation, i.e.\ the assumption on triviality of the centralizer
is not used in the proof of~\cite[Prop.\ 8.1]{EMV}. The claim thus follows from the results in \S\ref{sec:subgroups}.
\end{proof}

We also record the following.

\begin{lem}[Cf.~\cite{EMV}, Lemma 8.2]\label{lem:iterate}
There exists $\consta\label{k:iterate}>0$ with the following property.  
Let $S=P^+,$ $P^-,$ or~$G$ and suppose that $\mu$ is $\epsilon$-almost invariant under $S$
w.r.t.\ $\cS_d.$ Then 
\[
|\mu(q.f)-\mu(f)|\ll\epsilon |q|^{\ref{k:iterate}}\cS_d(f),\quad q\in S.
\]
\end{lem}

\begin{proof}
For $S=G,$ this is proved in~\cite[Lemma 8.2]{EMV}. 
Let us assume $S=P^+.$
In view of our assumption we have 
\beq\label{eq:S-alm-inv}
|\mu(q.f)-\mu(f)|\ll \epsilon\cS_d(f),\quad\text{for all $|q|\leq 2.$}
\eeq
In particular,~\eqref{eq:S-alm-inv} 
holds true for $q_{ij}=1+E_{ij}$ with $j>i.$ Let now $a\in H$ be a diagonal element 
with $|a|\ll t^\star$ so that 
\[
aq_{ij}a^{-1}=1+tE_{ij}=:q_{ij}(t).
\]
Since $\mu$ is invariant under $a_{ij}(t),$ the above,~\eqref{eq:S-alm-inv}
and properties of the Sobolev norm, see \S\ref{sec:sob}, imply that
\[
|\mu(q_{ij}(t).f)-\mu(f)|\ll t^\star\epsilon\cS_d(f)\ll |q_{ij}(t)|^\star\epsilon\cS_d(f).
\] 
Since $W$ is abelian we obtained the lemma for elements in $W.$
Since $\mu$ is $H$-invariant this gives the claim for $S=P^+.$ The proof for $S=P^-$ is similar.
\end{proof}

\section{Proof of Theorem \ref{thm:main}}

Recall that $\operatorname{d}\!w$ is the Lebesgue measure on $W\cong\bR^{kl}$, and for any $\tau>0$
we put $W[\tau]:=\{w\in w: \|w\|_{\infty}\leq\tau \}$. 
We let $m$ denote the $G$-invariant probability measure on $X$.

\begin{lem}\label{lem:mixing}
There exists a constant $\consta\label{exp:t}$ 
satisfying the following property.
Let $s\geq 1$, put $\tau=e^{s(k+l)}.$
Suppose that  
\[
a_{-s}z\in\Sfrak(R),
\]
for $z\in X$. Then for any $f\in C_c^\infty(X)$ we have
\begin{equation}\label{eq:mixing}
\left|\frac{1}{|W[\tau]|}\int_{W[\tau]} f(wz)\operatorname{d}\!w -\int_Xf\operatorname{d}\!m\right| \ll R^\star e^{-\ref{exp:t}s}\mathcal S_d(f).
\end{equation}
\end{lem}

\begin{proof}
By the definition of $\tau$,  note that $W[\tau]=a_{s}W[1]a_{-s}$.  
Denote $y:=a_{-s}z\in\mathfrak S(R).$
We have
\[
\frac{1}{|W[\tau]|}\int_{W[\tau]} f(wz) \operatorname{d}\!w=\frac{1}{|W[1]|}\int_{W[1]} f(a_{s}wy) \operatorname{d}\!w.
\]

Now using~\cite[Prop.~2.4.8]{KM1} (see also \cite[Thm.~2.3]{KM2} for the dependence on the height $R$) there exists a $\kappa>0$ so that the following holds:
\beq\label{eq:mixing-2}
\left|\frac{1}{|W[1]|}\int_{W[1]} f(a_{s}wy) \operatorname{d}\!w-\int_Xf \operatorname{d}\!m\right|\ll R^{\star}e^{-\kappa s}\mathcal S_d(f).
\eeq
\end{proof}

We will also need the following for the proof.
\begin{lem}\label{lem:orbit-faraway}
Suppose there exists some $\kappa>0$ so that
\[
{\mht}(Hx_0)\gg \vol(Hx_0)^{\kappa}. 
\]
Then Theorem~\ref{thm:main} holds (trivially).
\end{lem}    

\begin{proof}
If $\vol(Hx_0)<V$, the theorem holds trivially with $S=H$. 

On the other hand, our assumption
implies the theorem with $S=G$ if $V\leq \vol(Hx_0)$. 
Indeed, we may assume $\ref{height exponent}\geq\ref{vol exponent main}/\kappa$
so that 
\[
\mht(Hx_0)^{\ref{height exponent}} V^{-\ref{vol exponent main}}\gg \vol(Hx_0)^{\kappa\ref{height exponent}-\ref{vol exponent main}}\gg 1.
\]
This implies the conclusion of the theorem because of ($\cS$-2) in \S\ref{sec:sob}.
\end{proof}

\begin{proof}[Proof of Theorem~\ref{thm:main}]
Let $\mu$ denote the $H$-invariant probability measure on $Hx_0$ and $x_0=g_0\Gamma$. 
By Lemma \ref{lem:orbit-faraway} we may assume~\eqref{eq:low}.
Using Lemma~\ref{lem:eff-extrs-inv} we get almost invariance under an element in $\mathfrak r_1$. Then by Proposition~\ref{prop:eff-gen} and Lemma~\ref{lem:iterate} we get almost invariance under a subgroup $S=G,P^+,$ or $P^-.$  
Since the case $S=P^-$ is similar, we may assume $S=G$ or $P^+$ . Therefore, we assume throughout the argument that for any $f\in C_c^\infty(X)$ the following holds
\beq\label{eq:V-almost-inv}
|w_*\mu(f)-\mu(f)|\ll \epsilon S_d(f)\quad\text{for all $w\in W[\tau],$}
\eeq
with $\tau=\vol(Hx_0)^{\consta\label{k:tau-vol}}$ and $\epsilon=\vol(Hx_0)^{-\star}.$ 
Here $\ref{k:tau-vol}$ needs to be small enough to get~\eqref{eq:V-almost-inv}, we will need to optimize $\ref{k:tau-vol}$ further in the argument below.

We investigate  
\beq\label{eq:W-aver}
\frac{1}{|W[\tau]|}\int_{W[\tau]}\int_X f(wx)\operatorname{d}\!\mu(x)\operatorname{d}\!w.
\eeq
First note that~\eqref{eq:V-almost-inv} implies 
\beq
\label{eq:close-mu}
\left|\frac{1}{|W[\tau]|}\int_{W[\tau]}\int_X f(wx)\operatorname{d}\!\mu(x)\operatorname{d}\!w-\int_Xf \operatorname{d}\!\mu\right|  \ll \epsilon\mathcal S_d(f).
\eeq

Let $s$ be a parameter so that $\tau=e^{s(k+l)}$ and put $\mu_s=a_{-s}\mu.$ 
Apply Corollary~\ref{cor:closing-lem} to the measure $\mu_s$ and the closed orbit
\[
a_{-s}Hx_0=Ha_{-s}x_0.
\]
By the conclusion of that corollary there are two cases to consider.

{\it Case 1.} Assume Corollary~\ref{cor:closing-lem}(1) holds for $\mu_s.$ That is
\beq\label{eq:eq:mu-s-cusp-0}
\mu_{s}\pa{Ha_{-s}x_0\setminus \mathfrak S(R_0)}\leq 2^{-10}.
\eeq
For every $R$ put $B_{s,R}:=Hg_0\Gamma\setminus a_{s}\mathfrak S(R)$. Then by Lemma~\ref{lem:non-div}, 
for any $R>1$ we have 
\beq\label{eq:mu-s-cusp}
\mu\pa{B_{s,R}}=\mu_{s}\pa{Ha_{-s}x_0\setminus \mathfrak S(R)}\ll R^{-\ref{exp:escape}};
\eeq 
note that in view of~\eqref{eq:eq:mu-s-cusp-0}, we have $\mht(Ha_{-s}x_0)\ll R_0$ and
$R_0$ as chosen in Theorem \ref{thm:DM-parabolic} satisfies $R_0 \ll1$.

Let $R>R_0,$ using Fubini's theorem we can now rewrite~\eqref{eq:W-aver} in the form
\begin{align*}
\tfrac{1}{|W[\tau]|}\int_{W[\tau]}\int_X f(wx)\operatorname{d}\!\mu(x)\operatorname{d}\!w&=\tfrac{1}{|W[\tau]|}\int_X\int_{W[\tau]} f(wx)\operatorname{d}\!w\operatorname{d}\!\mu(x)\\
&=\tfrac{1}{|W[\tau]|}\int_{a_{s}\mathfrak S(R)}\int_{W[\tau]} f(wx)\operatorname{d}\!w\operatorname{d}\!\mu(x)\\
&\hspace{10mm}+\tfrac{1}{|W[\tau]|}\int_{B_{s,R}}\int_{W[\tau]} f(wx)\operatorname{d}\!w\operatorname{d}\!\mu(x).
\end{align*}
By \eqref{eq:mu-s-cusp} and property ($\cS$-2) of the Sobolev norm 
the second term above is $\ll \mathcal S_d(f)R^{-\star}.$ 
For the first term, note that by Lemma~\ref{lem:mixing} for any $z\in Hx_0\setminus B_{s,R}=Hx_0\cap a_s\mathfrak S(R)$
we have
\[
\left|\frac{1}{|W[\tau]|}\int_{W[\tau]} f(wz)\operatorname{d}\!w -\int_Xf \operatorname{d}\!m\right| \ll R^\star e^{-\ref{exp:t}s}\mathcal S_d(f).
\]
Hence using \eqref{eq:mu-s-cusp} one more time we get
\begin{multline*}
\left|\frac{1}{|W[\tau]|}\int_{a_{s}\mathfrak S(R)}\int_{W[\tau]} f(wx)\operatorname{d}\!w\operatorname{d}\!\mu(x)-
\int_Xf \operatorname{d}\!m\right|  \ll\\ R^\star e^{-\ref{exp:t}s}\mathcal S_d(f)+R^{-\star}\mathcal S_d(f)
\end{multline*}
Putting these together and recalling that $\tau=e^{\star s}$ we get
\begin{multline}
\label{eq:close-m}
\left|\frac{1}{|W[\tau]|}\int_{W[\tau]}\int_X f(wx)\operatorname{d}\!\mu(x)\operatorname{d}\!w-\int_Xf \operatorname{d}\!m\right|  \ll\\ R^\star \tau^{-\star}\mathcal S_d(f)+R^{-\star}\mathcal S_d(f).
\end{multline} 
Recall that so far our constraint on $\tau$ was only $\tau=\vol(Hx_0)^{\ref{k:tau-vol}}$ as in ~\eqref{eq:V-almost-inv}.
If we now choose $R= \vol(Hx_0)^\star$ and a small enough exponent we obtain
$R^\star \tau^{-\star}\ll \vol(Hx_0)^{-\star},$ 
then~\eqref{eq:close-mu} and~\eqref{eq:close-m} imply
\[
|\mu(f)-m(f)|\ll \vol(Hx_0)^{-\star}\mathcal S_d(f),
\]
and hence the theorem in this case.

We note that if the $\bQ$-rank of $\mathbb G$ is zero,
then we are always in case 1 of Corollary \ref{cor:closing-lem}.  

\medskip

{\it Case 2.} Recall again that $\mu_s$ is supported on $Ha_{-s}x_0=Ha_{-s}g_0\Gamma$. We now assume that Corollary~\ref{cor:closing-lem}(2) holds for $\mu_s$. In view of this assumption,
there exist $\lambda\in \Gamma\Xi$ and $\alpha\in\Delta$ such that 
\beq\label{eq:H-P+}
g_0^{-1}a_sHa_{-s}g_0=g_0^{-1}Hg_0\subset \lambda P^{(1)}_\alpha\lambda^{-1};
\eeq
moreover, there exists some $h_0\in H$ with $|h_0|\leq 2$ so that
\beq\label{eq:smallvol}
 d_\alpha(h_0a_{-s}g_0\lambda)<1.
\eeq

Recall from \S\ref{sec:DM} that $P_\alpha^{(1)}=\{g\in G:\vartheta_\alpha(g)v_\alpha=v_\alpha\}$,  where $v_\alpha$ corresponds to a rational subspace of $\gfrak$.
In view of~\eqref{eq:H-P+}, we have $Hg_0\lambda\subset g_0\lambda P^{(1)}_\alpha$; hence,    
\beq\label{eq:H-P+-again}
\mbox{$\vartheta_\alpha(hg_0\lambda)v_\alpha=\vartheta_\alpha(g_0\lambda)v_\alpha$ for all $h\in H$}.
\eeq
Using the definition of $\vartheta_\alpha$ and $v_\alpha$ again, we get from the above that 
\[
{\rm Ad}(g_0\lambda){\rm Lie}\pa{R_u(P_\alpha)}
\] 
is an invariant subspace for adjoint action of $H.$
Since $R_u(P_\alpha)$ is a unipotent group, this and the discussion in~\S\ref{sec:subgroups} imply that 
\beq\label{eq:had-to-ref}
\text{$g_0\lambda R_u(P_\alpha)\lambda^{-1} g_0^{-1}=W^+\;$ or $\;g_0\lambda R_u(P_\alpha)\lambda^{-1} g_0^{-1}=W^-$.}
\eeq  

We will consider these two subcases separately.

{\em Subcase 1.}\ Assume first that $g_0\lambda R_u(P_\alpha)\lambda^{-1} g_0^{-1}=W^-.$
We claim that under this assumption we have 
\beq\label{be:ht-tau}
{\mht}(Hx_0)\gg \tau^{\ref{k:stupid}}=\vol(Hx_0)^{\ref{k:tau-vol}\ref{k:stupid}}
\eeq
for some $\consta\label{k:stupid}>0$ depending only on $G.$ 
Note that~\eqref{be:ht-tau} implies the theorem in view of Lemma~\ref{lem:orbit-faraway}.

We now turn to the proof of~\eqref{be:ht-tau}. 
First note that since $g_0\lambda R_u(P_\alpha)\lambda^{-1} g_0^{-1}=W^-$ and both $\{a_s\}$ and $H$ normalize $W^-$,
we get that 
\beq\label{eq:normalize-W-}
\vartheta_\alpha(a_{-s}h_0g_0\lambda)v_\alpha\in \wedge^{\dim W^-}\Lie(W^-)
\eeq
and $\|\vartheta_\alpha(a_{-s}h_0g_0\lambda)v_\alpha\|=d_\alpha(a_{-s}h_0g_0\lambda)$.

We now have
\begin{align*}
\|\vartheta_\alpha(hg_0\lambda)v_\alpha\|&=\|\vartheta_\alpha(h_0g_0\lambda)v_\alpha\|&&\text{by~\eqref{eq:H-P+-again}}\\
&=\|\vartheta_\alpha(a_sa_{-s}h_0g_0\lambda)v_\alpha\|\\
&= e^{-\star s}\|\vartheta_\alpha(a_{-s}h_0g_0\lambda)v_\alpha\|&&\text{by~\eqref{eq:normalize-W-}}\\
&=e^{-\star s}d_\alpha(a_{-s}h_0g_0\lambda)\\
&\ll e^{-\star s}&&\text{by~\eqref{eq:smallvol}}.
\end{align*}

Recall from \S\ref{sec:DM} that $\vartheta(\lambda) v_\alpha$ corresponds to a rational subspace of $\gfrak$. 
Therefore, from the above we get that $\Ad(hg_0)\gfrak_{\mathbb Z}$ 
has a nontrivial sublattice with volume $\ll e^{-\star s}$ for every $h\in H$.
The claim in~\eqref{be:ht-tau} thus follows from the Minkowski's
theorem on successive minima in the geometry of numbers and since $e^{s(k+l)}=\tau=\vol(Hx_0)^{\ref{k:tau-vol}}$.

\medskip

{\em Subcase 2.}\ Assume now that $g_0\lambda R_u(P_\alpha)\lambda^{-1} g_0^{-1}=W^+$.
This assumption together with~\eqref{eq:H-P+} and the definitions of $P^+$, implies that 
\[
\lambda^{-1}g_0^{-1} P^+ g_0\lambda\subset P_\alpha^{(1)}=\{g\in G:\vartheta_\alpha(g)v_\alpha=v_\alpha\}.
\]
Hence we have
\beq\label{eq:P+-closed}
\lambda^{-1}g_0^{-1} P^+ g_0\lambda\text{  is the connected component of the identity in
$P_{\alpha}^{(1)}.$}
\eeq 
In particular, with $\lambda\in\Gamma\Xi\subset\mathbb G(\bQ)$ this implies 
that $P^+g_0\Gamma/\Gamma$ is a closed orbit. 
\begin{claim*} 
We have
\beq\label{eq:vol-P+}
\vol(P^+g_0\Gamma/\Gamma)\ll \tau^{\star}.
\eeq
\end{claim*}
Let us assume the claim and finish the proof.
We will show that 
\[
|\mu(f)-\mu_{P^+x_0}(f)|\ll \vol(Hx_0)^{-\star}\mathcal S_d(f)
\]
for any $f\in C^\infty_c(X)$, which will finish the proof.

The argument is based on finding a point which is {\em generic} both for $\mu$ and $\mu_{P^+x_0}$ unless 
we are in a situation where Lemma~\ref{lem:orbit-faraway} is applicable.
We first fix some notation.

Recall the definition of $R_{Hx_0}$ from~\eqref{eq:def-R0}.
 We apply the first part of the proof of Lemma~\ref{lem:separation}, up to and including
the estimate for the set $F$. For this part of the argument we only have to choose $T_0$
sufficiently large and $\delta$ has to satisfy the constraint $\delta\ll R_{Hx_0}^{-\star}$.
So we define $\delta=R_{Hx_0}^{-\star}$ appropriately and obtain the set $F\subset \mathfrak S(R_{Hx_0})$ be as in~\eqref{eq:def-F} for this $T_0$. 

Define
\beq\label{eq:choose-T0}
T_1:=\epsilon^{-\ref{k:gen-ni-tau}/2}=\vol(Hx_0)^{\star}
\eeq
where $\epsilon$ is as in~\eqref{eq:V-almost-inv}.
Recall that we may assume~\eqref{eq:low} by Lemma \ref{lem:orbit-faraway}, so that in particular $T_1>T_0$.

In view of the above claim, we will further assume that $\ref{k:tau-vol}$ in $\tau=\vol(Hx_0)^{\ref{k:tau-vol}}$ is small enough so 
that 
\beq\label{eq:T0-volP+x}
T_1^{-1}\vol(P^+x_0)<T_1^{-\star}.
\eeq 

We now turn to the construction of a generic point for $\mu$ and $\mu_{P^+x}$. 

Let $P^+_\delta:=\{\exp(v):v\in{\rm Lie}(P^+),\|v\|\leq\delta\}.$
Equip $P^+_\delta\times X$ with the product measure $m_{P^+}\times \mu$ where $m_{P+}$ is the Haar measure on $P^+$.
Now~\eqref{eq:V-almost-inv} together with an argument 
as in Lemma~\ref{lem:generic-almost-inv}(2), see also~\cite[Prop.\ 9.1]{EMV},
implies that the portion of points in 
$
P^+_\delta\times X,
$
so that $gx$ is {\em not} 
$[T_1,\epsilon^{-\ref{k:gen-ni-tau}}]$-generic for $\mu$
w.r.t.\ $\cS_{d'}$ is $\ll T_1^{-1}$.

Recall that $\mu(F)\geq 0.9.$
Therefore, using Fubini's Theorem, we get the following.
There exists a point $x_1\in F,$ which is $T_1$-generic for $\mu$ so that 
\beq\label{eq:alm-inv-tint}
m_{P^+}\left(|\{g\in P^+_\delta: gx_1\text{ is {\em not} $[T_1,\epsilon^{-\ref{k:gen-ni-tau}}]$-generic for $\mu$}\}\right)\ll T_1^{-1/2}m_{P^+}(P^+_\delta).
\eeq 
We may assume that the upper bound in~\eqref{eq:alm-inv-tint} is $<0.1m_{P^+}(P^+_\delta)$  (for otherwise
$\vol(Hx_0)$ is bounded).

Recall that $\mu_{P^+x_0}$ is the normalized probability $P^+$-invariant measure on $P^+x_0$. 
Applying Lemma~\ref{lem:generic-almost-inv}(1) with $\mu_{P^+x_0}$ we get that the $\mu_{P^+x_0}$ measure 
of the set of points which are {\em not} $T_1$-generic for $\mu_{P^+x_0}$ w.r.t.\ $\cS_d$ is $\ll T_1^{-1}$.
Therefore, taking the restriction of $\mu_{P^+x_0}$ to $P^+_\delta x_1\subset P^+x_1=P^+x_0$ we get the following.   
If $\mu_{P^+x_0}(P_\delta^+x_1)\gg T_1^{-1}$ (with a suitably chosen implicit multiplicative constant), then  
\beq\label{eq:local-Pdelta}
m_{P+}\left(\{g\in P^+_\delta: \text{$gx_1$
is $T_1$-generic for $\mu_{P^+x_0}$ w.r.t.\ $\cS_d$}\}\right)\geq 0.9m_{P^+}(P^+_\delta).
\eeq

In consequence, either~\eqref{eq:local-Pdelta} holds or $\delta$ is very small in the sense that $\mu_{P^+x_0}(P_\delta^+x_1)\ll T_1^{-1}$.
Let us first assume that the former holds. Then, in view of~\eqref{eq:alm-inv-tint} and~\eqref{eq:local-Pdelta}, we may replace 
$x_1$ by $gx_1$ for some $g\in P^+_\delta$ so that 
$gx_1$ is $[T_1,\epsilon^{-\ref{k:gen-ni-tau}}]$-generic for $\mu,$ moreover, $gx_1$ is 
$T_1$-generic for $\mu_{P^+x_0}.$ This, in particular, implies that
\beq \label{eq:maincase}
|\mu(f)-\mu_{P^+x_0}(f)|\ll T^{-1} \mathcal S_d(f)
\eeq
for all $T\in [T_1,\epsilon^{-\ref{k:gen-ni-tau}}].$ 

In view of~\eqref{eq:choose-T0},~\eqref{eq:maincase} completes the proof in this case.

It thus remains to consider the that $\mu_{P^+x_0}(P_\delta^+x_1)\ll T_1^{-1}$. 
First note that this is to say $m_{P+}(P^+_\delta)\ll T_1^{-1}\vol(P^+x_0)$. This and~\eqref{eq:T0-volP+x} imply that
there exists some $\consta\label{k:stupid-2}$ so that
\[
T_1^{\ref{k:stupid-2}}=\vol(Hx_0)^\star\ll \mht(Hx_0)=\delta^{-\star}. 
\]
This again implies the theorem by Lemma~\ref{lem:orbit-faraway}.

The proof thus is complete modulo the claim in~\eqref{eq:vol-P+}. 
\end{proof}

\begin{proof}[Proof of~Claim]
Recall from~\eqref{eq:P+-closed}
that $\lambda^{-1}g_0^{-1} P^+ g_0\lambda=P_{\alpha}^+$  is the connected component of the identity in
$P_{\alpha}^{(1)}$ and  that $\Xi\subset\mathbb G(\bQ)$ is a finite set.
For any $\xi\in\Xi$, define $P_{\alpha,\xi}:=\xi P_\alpha\xi^{-1}$, $P^{(1)}_{\alpha,\xi}=\xi P_\alpha^{(1)}\xi^{-1},$  $P^{+}_{\alpha,\xi}=\xi P_\alpha^{+}\xi^{-1},$ and $v_{\alpha,\xi}=\vartheta_\alpha(\xi) v_\alpha$. 
Let $A_{\alpha,\xi}$ denote the center of the Levi component of $P_{\alpha,\xi}.$
Then $A_{\alpha,\xi}P^+_{\alpha,\xi}$ has finite index in $P_{\alpha,\xi}.$ Let $B_{\alpha,\xi}\subset P_{\alpha,\xi}$ denote a set of representative for $P_{\alpha,\xi}/A_{\alpha,\xi}P^+_{\alpha,\xi}.$

Write $\lambda=\gamma\xi\in\Gamma\Xi.$
Since $G={\rm SO}(k+\ell)P_{\alpha,\xi}$ the above discussion implies that we may write
\[
g_0\gamma=cba_\gamma p^+
\] 
where $c\in{\rm SO}(k+\ell),$ $b\in B_{\alpha,\xi},$ $a_\gamma\in A_{\alpha,\xi},$ and $p^+\in P_{\alpha,\xi}^+$.
In particular, since $P^+g_0\Gamma=g_0\gamma P^+_{\alpha,\xi}\Gamma,$ we have
\beq\label{eq:vol-a}
\vol(P^+g_0\Gamma/\Gamma)\asymp\vol(a_\gamma P^+_{\alpha,\xi}\Gamma/\Gamma).
\eeq
Therefore, it suffices to bound $\vol(a_\gamma P^+_{\alpha,\xi}\Gamma/\Gamma).$

The argument is similar to the one in {\em Subcase 1}. Note first that by \eqref{eq:smallvol} we have
\[
\|\vartheta_\alpha(a_{-s}h_0g_0\lambda)v_\alpha\|=d_\alpha(a_{-s}h_0g_0\lambda)<1.
\]
for some $|h_0|\leq 2$. Also recall that $g_0\lambda R_u(P_\alpha)\lambda^{-1} g_0^{-1}=W^+$.
Therefore arguing as in {\em Subcase 1}, with $W^+$ in place of $W^-$, we have 
\[
\|\vartheta_\alpha(h_0g_0\lambda)v_\alpha\|=\|\vartheta_\alpha(a_sa_{-s}h_0g_0\lambda)v_\alpha\|\ll \tau^\star.
\] 
Now,  in view of the fact that $P^+_{\alpha,\xi}\subset P^{(1)}_{\alpha,\xi}$, the above discussion implies
\begin{align}\label{eq:bdd-height}
\|\vartheta_\alpha(a_\gamma)v_{\alpha,\xi}\|&\asymp\|\vartheta_\alpha(g_0\gamma)v_{\alpha,\xi}\|\\
\notag&=\|\vartheta_\alpha(g_0\gamma\xi)v_{\alpha}\|\\
&=\|\vartheta_\alpha(h_0g_0\lambda)v_{\alpha}\|\ll \tau^\star,
\end{align} 
 where \eqref{eq:had-to-ref} and $\lambda=\gamma\xi$ where used in the last equality.  This implies that $|a_{\gamma}|\ll\tau^\star.$ 
Now since the distortion of the volume when applying $g$ is bounded by $|g|^\star$ we get~\eqref{eq:vol-P+} from this bound and~\eqref{eq:vol-a}.
\end{proof}

\bibliographystyle{natbib}
\bibliography{Eff}

\end{document}